\documentclass[a4paper,10pt]{amsart}

\usepackage[margin=1in]{geometry}
\usepackage{amsmath,amssymb,mathtools}
\usepackage{microtype}
\usepackage[colorlinks=true,linkcolor=blue,citecolor=blue,urlcolor=blue]{hyperref}

\newtheorem{theorem}{Theorem}[section]
\newtheorem{lemma}[theorem]{Lemma}
\newtheorem{proposition}[theorem]{Proposition}
\newtheorem{corollary}[theorem]{Corollary}
\theoremstyle{remark}

\newcommand{\N}{\mathbb N}
\newcommand{\Q}{\mathbb Q}
\newcommand{\Z}{\mathbb Z}
\newcommand{\Int}{\operatorname{Int}}
\newcommand{\Span}{\operatorname{span}}

\title{Sparse polynomial-weighted expansions}
\author{Han Wang}
\email{han@hanziwww.com}
\date{}
\hypersetup{
  pdftitle={Sparse polynomial-weighted expansions},
  pdfauthor={Han Wang},
  pdfsubject={Polynomial-window density for rational weighted fixed-base expansions},
  pdfkeywords={irrationality, fixed-base expansions, polynomial weights,
    sparse sequences, carry recurrence, Erdos Problem 260}
}

\begin{document}

\begin{abstract}
Let $b\ge2$ be an integer, let $p\in\Q[x]$ be nonzero, and let
$S\subseteq\N$ be infinite. We prove that if
$\sum_{n\in S}p(n)b^{-n}$ is rational, then, for every fixed
$0<\theta\le1$, there is a constant $c>0$ such that
$|S\cap(N,N+W]|\ge cW$ for all sufficiently large $N$ and every
$N^\theta\le W\le N$. Thus rationality forces positive lower density; if
$S=\{a_1<a_2<\cdots\}$ and $d=\deg p$, it also forces
$a_{j+1}-a_j\le d\log_ba_j+O(1)$. In the binary linear case, the result
implies the irrationality conjectured by Erd\H{o}s whenever $a_j/j\to\infty$.
The proof turns rationality into an integer carry orbit of polynomial height.
Repeated gap words lock the orbit onto rational polynomial graphs, while the
normalized highest Newton coefficient evolves by an expanding affine map.
Denominator preservation, polynomial sampling, and an interior--exterior
count then rule out sparse polynomial-scale windows.
\end{abstract}

\maketitle

\section{Introduction}

Throughout, $\N=\{1,2,\ldots\}$. A rational number has an eventually periodic
canonical expansion in every integer base. That elementary characterization
does not directly decide the rationality of a series whose coefficients are
not constrained to the base-$b$ digit alphabet: carries may propagate across
long intervals and substantially change the visible support. For an infinite set
$S\subseteq\N$ and a nonzero polynomial $p\in\Q[x]$, consider the absolutely
convergent series
\[
 \Phi_{b,p}(S)=\sum_{n\in S}p(n)b^{-n},\qquad b\ge2.
\]
We ask how sparse $S$ can be when $\Phi_{b,p}(S)$ is rational. The answer is
local: at every sufficiently large location, and on every scale between a
fixed positive power of that location and the location itself, $S$ must
occupy a positive proportion of the interval.

The arithmetic of rapidly convergent series has traditionally been studied
by imposing integrality on suitably scaled tails. Oppenheim developed such
criteria for Cantor-series-type expansions~\cite{Oppenheim1954}, and
Erd\H{o}s and Straus obtained rationality criteria for product-denominator
series in terms of eventual integer recurrences~\cite{ErdosStraus}. Erd\H{o}s
also established irrationality results for sparse fixed-base
series~\cite{Erdos1957}. More recently, Han\v{c}l and Tijdeman treated
polynomial numerators over polynomially generated cumulative
denominators~\cite{HanclTijdeman}. These works exploit the arithmetic of the
successive denominators. In the present problem the denominator ratio is the
fixed integer $b$, whereas the arbitrary support $S$ carries the relevant
combinatorial information.

A different rigidity principle comes from the complexity of canonical digit
strings. Adamczewski and Bugeaud proved strong lower bounds for the block
complexity of integer-base expansions of irrational algebraic
numbers~\cite{AdamczewskiBugeaud}. Here, however, the sequence
$p(n)\mathbf1_S(n)$ is unbounded unless $p$ is constant, is not a canonical
digit sequence, and has no assumed finite descriptive structure. After carry
normalization, the canonical digits need not retain the geometry of $S$.
Digit-complexity criteria therefore do not directly address the question
above.

The motivating special case is $b=2$ and $p(x)=x$. If
$S=\{a_1<a_2<\cdots\}$, Erd\H{o}s asked whether
\[
 \sum_{j\ge1}a_j2^{-a_j}
\]
must be irrational whenever $a_j/j\to\infty$; see Erd\H{o}s and
Graham~\cite{ErdosGraham}. Erd\H{o}s proved irrationality under the stronger
condition $a_{j+1}-a_j\to\infty$~\cite{Erdos1981}. Borwein and Loring later
proved it when
\[
 \limsup_{j\to\infty}
 \frac{a_{j+1}-a_j}{\log a_{j+1}}=\infty
\]
and developed a base-change algorithm for representations of this
form~\cite{BorweinLoring}. Neither gap condition follows from
$a_j/j\to\infty$: that hypothesis gives global average sparsity but permits
long clusters of short gaps. The obstruction is therefore genuinely local.

Theorem~\ref{thm:main} replaces these exceptional-gap hypotheses by a
uniform density conclusion. For every fixed $0<\theta\le1$, rationality of
$\Phi_{b,p}(S)$ implies
\[
 |S\cap(N,N+W]|\ge cW
 \qquad (N^\theta\le W\le N)
\]
for all sufficiently large $N$, with one constant $c>0$ valid throughout
the displayed range of $W$. In particular, rationality forces positive lower
density and logarithmic support gaps. The binary linear case settles the
question above, while the general statement simultaneously allows every
integer base and every nonzero rational polynomial weight. The windowwise
conclusion is strictly more informative than the global density obstruction
needed for the Erd\H{o}s problem.

The passage from linear to polynomial weights changes the carry geometry.
After an integral normalization, rationality produces a positive integer
orbit $R_N$ satisfying
\[
 R_{N+1}=bR_N-Qw(N+1)\mathbf1_S(N+1),\qquad R_N\ll N^d.
\]
A repeated logarithmic word of support gaps then forces the corresponding
carry endpoints onto a single rational polynomial graph of degree at most
$d=\deg p$. If $\mu$ is its normalized highest Newton coefficient, appending
a gap $g$ induces the exact dynamics
\[
 \mu\longmapsto b^g\mu-1.
\]
This separates the argument into an expanding exterior regime and a narrow
interior regime with at most one admissible successor gap.

On a window $(N,N+W]$, a double count supplies a mass of forward gap words of
order $mW$. Exterior words are controlled by exponential growth of their
leading coefficient and a polynomial sublevel estimate. In the interior,
the part of the leading-state denominator coprime to $b$ is preserved. It
controls the mean gap and permits a decomposition into short canonical
blocks. A block word determines the exact leading state; frequent block-end
graphs coalesce after a common continuation, while an integral fibre
estimate controls their endpoints. The interior and exterior bounds are
incompatible with a sparse window.

The lower cutoff $W\ge N^\theta$ enters only when the subpower word counts and
endpoint errors are made uniform in the window. The resulting estimates do
not reach polylogarithmic windows, and no optimality of this restriction is
asserted. Section~2 develops the carry and locking mechanisms; Section~3
establishes the continuation counts and completes the window argument.

\begin{theorem}[Polynomial-window density]\label{thm:main}
Fix an integer $b\ge2$, let $p\in\Q[x]$ be nonzero, and let
$S\subseteq\N$ be infinite.
If
\[
 \eta=\sum_{n\in S}p(n)b^{-n}
\]
is rational, then for every fixed $0<\theta\le1$ there are constants $c>0$
and $N_0$ such that
\[
 |S\cap(N,N+W]|\ge cW
\]
for all integers $N,W$ with $N\ge N_0$ and
\[
 N^\theta\le W\le N.
\]
Let $A_p$ be the least positive integer with $A_p p\in\Z[x]$. The constant
$c$ may be chosen to depend only on $b,p,\theta$ and the part coprime to $b$
of the reduced
denominator of $A_p\eta$. The cutoff $N_0$ may also depend on $S$ and the full
rational value. Neither constant depends on $W$.
\end{theorem}

\begin{corollary}[Density and gaps]\label{cor:erdos}
Fix $b\ge2$, let $p\in\Q[x]$ be nonzero of degree $d$, let
$S\subseteq\N$ be infinite, and set
\[
 \eta=\sum_{n\in S}p(n)b^{-n}.
\]
If $\eta$ is rational, then
\[
 \underline d(S):=\liminf_{x\to\infty}\frac{|S\cap[1,x]|}{x}>0.
\]
Writing $S=\{a_1<a_2<\cdots\}$, one also has
\[
 a_{j+1}-a_j\le d\log_ba_j+O_{b,p,\eta}(1).
\]
For every $\varepsilon>0$, rationality also gives
\[
 a_{j+1}-a_j\le a_j^\varepsilon
\]
for all sufficiently large $j$. Consequently, if
$\underline d(S)=0$ (equivalently,
$\limsup_{j\to\infty}a_j/j=\infty$), then $\eta$ is irrational.
Moreover, rationality implies
\[
 a_{j+1}-a_j=a_j^{o(1)}.
\]
\end{corollary}

\begin{proof}
Apply Theorem~\ref{thm:main} with $\theta=1$ and $W=N$. For large $x$, taking
$N=\lfloor x/2\rfloor$ gives
$|S\cap[1,x]|\ge|S\cap(N,2N]|\ge cN$, and hence
$\underline d(S)>0$. For an increasing enumeration of $S$, positive lower
density is equivalent to $a_j\ll j$. The gap bound is the last assertion of
Lemma~\ref{lem:carries}; the normalization below changes only finitely many
indices. For every $\varepsilon>0$, one has
$d\log_b x+O(1)\le x^\varepsilon$ for all sufficiently large $x$, which gives
the remaining gap assertion.
\end{proof}

\paragraph{Example.}
For $q\ge1$, the set $S=q\N$ satisfies
\[
 \sum_{n\in S}p(n)b^{-n}
 =\sum_{k\ge1}p(qk)(b^{-q})^k\in\Q,
\]
since the generating function of a polynomial sequence is rational. Here
$\underline d(S)=1/q$, so the positive constant in
Theorem~\ref{thm:main} cannot be uniform over all rational values.

Let $A_p$ be as in Theorem~\ref{thm:main}. After changing sign and deleting
finitely many indices, it is enough to consider the integer-valued polynomial
\[
 w(n)=\sum_{j=0}^d\omega_j\binom nj,\qquad
 \omega_j\in\Z,\qquad \omega_d>0,
\]
which is positive on the remaining tail. These operations preserve all
eventual density conclusions. The resulting series differs from
$\pm A_p\eta$ by a rational number whose denominator divides a power of $b$.
Write the resulting rational series as $P/Q_{\rm full}$ in lowest terms and
put
\[
 Q_{\parallel}=\prod_{\ell\mid b}\ell^{v_\ell(Q_{\rm full})},
 \qquad Q=Q_{\rm full}/Q_{\parallel}.
\]
The product is over primes $\ell$ dividing $b$, so $(Q,b)=1$. All implicit
constants below may depend on $b,w,Q,\theta$, but not on $N$ or $W$; the
discarded range may depend on $Q_{\parallel}$. Finite deletion can only replace
the coprime part of the denominator of $A_p\eta$ by a divisor. Taking the
minimum over these finitely many divisors gives the dependence asserted in
Theorem~\ref{thm:main}.

\section{Carries and polynomial-scale windows}

Write $d_n=\mathbf1_S(n)$ and, for all sufficiently large $N$, define
\[
 R_N=Q\sum_{j\ge1}w(N+j)d_{N+j}b^{-j}.
\]
List the support as $a_1<a_2<\cdots$ and put $g_i=a_{i+1}-a_i$.

\begin{lemma}[Carries and gaps]\label{lem:carries}
For all sufficiently large $N$,
\[
 R_N\in\Z,\qquad
 R_{N+1}=bR_N-Qw(N+1)d_{N+1},
\]
and
\[
 1\le R_N\le C(N+1)^d.
\]
Moreover,
\[
 g_i\le d\log_ba_i+C
\]
for all sufficiently large $i$.
\end{lemma}

\begin{proof}
Choose $N_*$ so that $Q_{\parallel}\mid b^{N_*}$. For $N\ge N_*$,
\[
 Q_{\parallel}R_N
 =Pb^N-Q_{\rm full}\sum_{n\le N}w(n)d_nb^{N-n},
\]
and both terms on the right are divisible by $Q_{\parallel}$: the first by
the choice of $N_*$, the second because $Q_{\parallel}\mid Q_{\rm full}$.
Division by $Q_{\parallel}$ gives $R_N\in\Z$; the recurrence follows from the
definition. Positivity of the tail and infinitude of $S$ give $R_N\ge1$, while
\[
 R_N\le Q\sum_{j\ge1}w(N+j)b^{-j}\ll(N+1)^d.
\]
If $x=a_i$ and $x+g_i=a_{i+1}$, then
$R_{x+g_i-1}=b^{g_i-1}R_x\ge b^{g_i-1}$. Hence
\[
 b^{g_i-1}\ll(x+g_i)^d.
\]
The displayed inequality excludes $g_i\ge x$ for large $x$; hence
$x+g_i<2x$, which gives the logarithmic bound.
\end{proof}

If $d=0$, Lemma~\ref{lem:carries} gives $g_i\le C$ for all sufficiently large
$i$. Any interval of length $W$ beyond this finite range then contains at
least $W/C-O(1)$ support points. Since $W\ge N^\theta\to\infty$, this proves
Theorem~\ref{thm:main} in degree zero. Assume henceforth that $d\ge1$ and put
\[
 s=\frac{d(d+1)}2,\qquad \sigma=\frac1s.
\]

Fix an admissible interval $I=(N,N+W]$ and set
\[
 L=\lceil\log_bN\rceil,\qquad
 M=|S\cap I|,\qquad
 \delta=\frac MW,\qquad
 m=\left\lfloor L\sqrt\delta\right\rfloor.
\]
For large $N$, Lemma~\ref{lem:carries} gives $\delta>0$. Suppose that
$\delta\le\delta_0$, where $\delta_0\le1/4$ will be chosen after all fixed
constants.

\begin{lemma}[The local scale]\label{lem:localscale}
Uniformly for $N^\theta\le W\le N$,
\[
 \delta\gg\frac1L,\qquad m\to\infty,\qquad m<M.
\]
If $a_k\in I$ and $0\le r\le m$, then
\[
 a_{k+r}\le N+W+O(L^2)=N+W+o(W),
\]
and every gap occurring in these windows is at most $dL+C$.
\end{lemma}

\begin{proof}
Applying Lemma~\ref{lem:carries} to the gaps straddling the endpoints of $I$
places its first and last support points within $O(L)$ of them and bounds all
intervening gaps by $O(L)$. Since
$W\ge N^\theta\gg L$,
\[
 M\gg\frac WL,\qquad \delta\gg\frac1L.
\]
Thus $m\to\infty$ and $M/m\gg W/L^{3/2}\to\infty$, uniformly in $W$.
If one of the next $m$ points after $a_k\in I$ were the first above $3N$,
Lemma~\ref{lem:carries} would bound its displacement from $a_k\le2N$ by
$O(mL)=O(L^2)=o(N)$, a contradiction. Thus all these points lie below $3N$;
their gaps are at most $dL+C$ and their total displacement is
$O(mL)=O(L^2)=o(W)$.
\end{proof}

For each $a_k\in I$, consider the next $m$ gaps and write
\[
 \mathbf g_k=(g_k,\ldots,g_{k+m-1}),\qquad
 V_k=a_{k+m}-a_k.
\]

\begin{lemma}[Window mass and word complexity]\label{lem:windows}
Uniformly for $N^\theta\le W\le N$,
\[
 \sum_{a_k\in I}V_k\ge(1-o(1))mW,\qquad
 V_k\le m(dL+C).
\]
For every fixed $C_1>0$ and
$0<\rho_0\le\min\{1/2,C_1/2\}$, the number of positive integer
words with at most $m\le\rho_0L$ entries and span at most $C_1L$ is at most
\[
 N^{\omega_{C_1}(\rho_0)+o(1)},
\]
where
\[
 \omega_{C_1}(\rho)
 =\frac{\rho}{\log b}\log\frac{e(C_1+1)}{\rho},
 \qquad
 \omega_{C_1}(\rho_0)\longrightarrow0
 \quad(\rho_0\downarrow0).
\]
All $o(1)$ terms are uniform over the admissible intervals.
\end{lemma}

\begin{proof}
Let $a_i$ be the first support point in $I$. Reordering the double sum gives
\[
 \sum_{a_k\in I}V_k
 =\sum_{r=0}^{m-1}(a_{i+M+r}-a_{i+r})
 \ge mW-O(m^2L).
\]
Since $mL/W\le L^2/W=o(1)$ uniformly, the first assertion follows; the
second follows from Lemma~\ref{lem:localscale}. Finally,
\[
 \sum_{r\le m}\binom{\lceil C_1L\rceil}{r}
 \le m\left(\frac{eC_1L+e}{m}\right)^m.
\]
The function
\[
 x\longmapsto x\log\frac{e(C_1+1)}{x}
\]
is increasing on $(0,C_1/2]$. Since $L=(\log N)/(\log b)+O(1)$,
division by $\log N$ gives
the stated $\omega_{C_1}$; the errors from the integer parts are uniform in
$N^\theta\le W\le N$.
\end{proof}

Set
\[
 \Lambda_{\rm lock}=d+s+2,\qquad
 \Lambda_{\rm ext}=d+s+2,\qquad
 \Lambda_{\rm cont}=2s+\frac{3d}{2}+2.
\]
Choose $C_0$ sufficiently large in terms of these constants and $b,w,Q$.
A window is \emph{short} if $V_k\le C_0L$. Since
$m\ge\frac12L\sqrt\delta$ for large $N$,
\begin{equation}\label{eq:short}
 \sum_{V_k\le C_0L}V_k
 \le C_0LM
 \le2C_0\sqrt{\delta_0}\,mW.
\end{equation}

For a long window, take the shortest initial gap word whose span $G$ exceeds
$\Lambda_{\rm lock}L$. Then
\begin{equation}\label{eq:prefixspan}
 \Lambda_{\rm lock}L<G
 \le(\Lambda_{\rm lock}+d)L+C,\qquad
 V_k-G\ge\frac34V_k.
\end{equation}
Lemma~\ref{lem:windows} gives
\begin{equation}\label{eq:prefixcount}
 \mathcal P_N\le N^{\varepsilon_p+o(1)},\qquad
 \varepsilon_p
 =\omega_{\Lambda_{\rm lock}+d+1}(\sqrt{\delta_0})
 \longrightarrow0\quad(\delta_0\downarrow0)
\end{equation}
possible prefixes. Call a prefix rare if it occurs at most $d+1$ times.
The rare-prefix windows contribute
\begin{equation}\label{eq:rare}
 O_d(\mathcal P_NmL)=o(mW)
\end{equation}
once $\varepsilon_p<\theta$, since
$N^{\varepsilon_p+o(1)}L/W\le
N^{\varepsilon_p-\theta+o(1)}L$.

Fix a nonrare prefix $p=(h_1,\ldots,h_r)$, put
$G_j=h_1+\cdots+h_j$, $G=G_r$, and let $y=x+G$ be its terminal coordinate.
Iteration of the carry recurrence gives
\begin{equation}\label{eq:wordrec}
 R_y+QF_p(y)=b^GR_x,\qquad
 F_p(y)=\sum_{j=1}^rb^{G-G_j}w\bigl(y-(G-G_j)\bigr).
\end{equation}
The polynomial $F_p$ is integer-valued and has degree at most $d$.

\begin{lemma}[Polynomial locking]\label{lem:locking}
All terminal points of a nonrare prefix lie on one graph
\[
 R=P_p(y),\qquad P_p\in\Q[y],\qquad \deg P_p\le d.
\]
Moreover, for some integer $T_p$,
\begin{equation}\label{eq:graphden}
 1\le T_p\ll_d W^s,\qquad
 T_pP_p\in\Int_d(\Z),
\end{equation}
where $\Int_d(\Z)$ is the set of integer-valued polynomials of degree at
most $d$.
\end{lemma}

\begin{proof}
For any $d+2$ distinct occurrences, form the integer determinant with rows
\[
 \left(1,\binom y1,\ldots,\binom yd,R_y\right).
\]
Since $F_p\in\Int_d(\Z)$, write
\[
 QF_p(y)=\sum_{j=0}^dc_j\binom yj,\qquad c_j\in\Z,
\]
and add $c_j$ times the column $\binom yj$ to the last column for
$0\le j\le d$. By~\eqref{eq:wordrec}, its entries become $b^GR_x$; the
determinant is therefore divisible by $b^G$. Expanding this unchanged
determinant in its original last column and using the Vandermonde formula give
\[
 |\det|\ll_d N^dW^s.
\]
Here all terminal coordinates lie in an interval of length $W+o(W)$ and are
$O(N)$. Since $G>(d+s+1)L$ and $W\le N$, the only multiple of $b^G$ satisfying
this bound is zero for large $N$. Fixing $d+1$ terminal points shows that every
other terminal point lies on their unique interpolation graph of degree at
most $d$. Applying Cramer's rule to the integer matrix
$(\binom{y_i}{j})_{0\le i,j\le d}$ at $d+1$ such points gives
\eqref{eq:graphden}, since
\[
 \det\left(\binom{y_i}{j}\right)_{0\le i,j\le d}
 =\frac{\prod_{0\le i<j\le d}(y_j-y_i)}{\prod_{j=1}^d j!}
 \ll_d W^s.
\]
\end{proof}

Write
\[
 P_p(x)=\sum_{j=0}^d\vartheta_j\binom xj,\qquad
 A=Q\omega_d,\qquad \mu=\frac{\vartheta_d}{A}.
\]
Appending a common gap $g$ sends the graph to
\[
 P_p^{(g)}(y)=b^gP_p(y-g)-Qw(y),
\]
and therefore
\begin{equation}\label{eq:topmap}
 \mu\longmapsto b^g\mu-1.
\end{equation}

Call a state \emph{interior} if $0<(b-1)\mu<1$, and \emph{strictly
exterior} if $\mu<0$ or $(b-1)\mu>1$.

\begin{lemma}[Interior--exterior dichotomy]\label{lem:dichotomy}
From an interior state at most one gap leads to another interior state.
Every nonrare long window has either an interior segment of span at least
$V_k/4$ or a strictly exterior segment of span at least $V_k/4$ after its
initial prefix.
\end{lemma}

\begin{proof}
The first assertion is equivalent to
\[
 1<b^g\mu<\frac b{b-1},
\]
which admits at most one power of $b$. The lower exterior region is forward
invariant. For $E=(b-1)\mu-1>0$,
\[
 E'=b^gE+(b^g-b)\ge b^gE,
\]
so the upper exterior region is also forward invariant. The lower boundary
exits at the next gap; on the upper boundary only a unit gap remains on the
boundary. The total span spent on boundary states is therefore $O(L)+m$.
In view of~\eqref{eq:prefixspan}, enlarging $C_0$ and then decreasing
$\delta_0$ ensures an exterior span of at least $V_k/4$ whenever the interior
span is less than $V_k/4$.
\end{proof}

Classify a nonrare long window as interior if its interior segment has span
at least $V_k/4$, and as exterior otherwise.

\section{Counting the continuations}

\begin{proposition}[Exterior contribution]\label{prop:exterior}
The total span of the exterior windows is $o(mW)$, uniformly for
$N^\theta\le W\le N$.
\end{proposition}

\begin{proof}
Since $T_p\ll W^s$ and a fixed multiple of $T_p$ clears both exterior
displacements after every graph transform, at the first strict exterior state
either
$-\mu\ge cW^{-s}$ or $E=(b-1)\mu-1\ge cW^{-s}$. After a fixed
strictly exterior continuation of span $F$,~\eqref{eq:topmap} and the
preceding formula for $E'$ make the absolute leading monomial coefficient of
the transformed graph at least
\[
 cb^FW^{-s}.
\]
If $H$ is a degree-$d$ real polynomial with leading coefficient $a_d\ne0$,
then factorization over $\mathbb C$ gives
\begin{equation}\label{eq:sublevel}
 \#\{n\in J\cap\Z:|H(n)|\le Y\}
 \le d+2d(Y/|a_d|)^{1/d}.
\end{equation}
Actual carry values are $O(N^d)$, so a fixed exterior word of span $F$
leaves at most
\[
 O_d\left(1+\left(\frac{N^dW^s}{b^F}\right)^{1/d}\right)
\]
possible terminal coordinates.

Fix a repeated prefix and its locked graph. For each occurrence record
\begin{enumerate}
\item the number $\nu\le m$ of successive gaps for which the state remains
interior;
\item whether the first noninterior state is lower exterior, upper exterior,
the lower boundary, or the upper boundary;
\item the length $u\le m$ of a possible unit-gap run on the upper boundary;
and
\item up to two transition gaps: the gap by which the interior is left, if
present, and the gap by which a boundary state is left.
\end{enumerate}
The two gaps in the last item are at most $dL+C$. Thus the number of records is
at most $O_d((m+1)^2(L+1)^2)=L^{O_d(1)}$. The initial state is fixed by the
prefix graph. Interior uniqueness determines its first $\nu$ gaps, and the
record then determines the continuation up to the first strictly exterior
state.

Starting there, select the shortest word whose span $F$ exceeds
$\Lambda_{\rm ext}L$. Its last gap is at most $dL+C$, and hence
\[
 \Lambda_{\rm ext}L<F\le(\Lambda_{\rm ext}+d)L+C.
\]
The exterior segment has span at least $V_k/4$, so this word is present once
$C_0$ is large enough. Lemma~\ref{lem:windows} gives at
most $N^{\varepsilon_e+o(1)}$ selected words, where
\[
 \varepsilon_e=
 \omega_{\Lambda_{\rm ext}+d+1}(\sqrt{\delta_0})
 \longrightarrow0\qquad(\delta_0\downarrow0).
\]
For a fixed prefix, record, and selected word, the transformed graph is fixed.
Since $F>\Lambda_{\rm ext}L$ and $W\le N$, the sublevel bound above leaves
$O(1)$ possible terminal coordinates. The code consisting of the prefix, the
record, the selected word, and that coordinate is injective: subtracting the
recorded continuation span and the prefix span from the terminal coordinate
recovers the window start. Hence, using $V_k\ll mL$ and absorbing the record
count into $N^{o(1)}$,
\[
 \sum_{k\ {\rm exterior}}V_k
 \ll N^{\varepsilon_p+\varepsilon_e+o(1)}mL=o(mW)
\]
after choosing $\varepsilon_p+\varepsilon_e<\theta$.
\end{proof}

Consider an interior window. At the beginning of its interior segment the
reduced denominator of $\mu$ is $O(W^s)$ by~\eqref{eq:graphden}. Let $q_b$
be its $b$-primary part and $t$ the least integer such that $q_b\mid b^t$.
We have
\[
 t=O_b(\log q_b)=O_{b,d}(L).
\]
The interior segment has span at least $V_k/4>C_0L/4$. Hence, after fixing
$C_0$ sufficiently large, it contains a shortest initial segment of span
$h\ge t$. Its last gap is at most $dL+C$, so $h=O_{b,d}(L)$.

Write the initial state in lowest terms as $\mu=a/(q_bq)$, where $(q,b)=1$.
Iteration of~\eqref{eq:topmap} gives
\[
 \mu'=\frac{(b^h/q_b)a-Mq}{q}
\]
for some $M\in\Z$. Both $a$ and $b^h/q_b$ are coprime to $q$, so this fraction
is reduced. The same calculation at every later interior step shows that the
coprime denominator $q$ is preserved. Moreover, $q>1$, since the interior
interval contains no integer. The removed span is $O_{b,d}(L)$; increasing
$C_0$ once more leaves a segment of span
\[
 G\ge V_k/8.
\]
This segment contains $n\ge1$ gaps.
Writing its states as $\mu_j=r_j/q$ gives
\begin{equation}\label{eq:denrec}
 0<r_j<\frac q{b-1},\qquad
 b^{g_j}r_{j-1}=q+r_j.
\end{equation}
With $n\le m$ and mean gap $z=G/n$,
~\eqref{eq:denrec} and $r_{j-1}\ge1$ give $b^{g_j}\le C_bq$.
Multiplication over the $n$ gaps yields $b^G\le(C_bq)^n$. Together with
$G\ge V_k/8$, this gives
\begin{equation}\label{eq:qz}
 q\gg_b b^z,\qquad g_j\le\log_b(C_bq),\qquad
 z\ge\frac{V_k}{8m}\gg\frac{C_0}{\sqrt{\delta_0}}.
\end{equation}

Choose dyadic $D,Z$ with $D\le q<2D$ and $Z\le z<2Z$. Then
$D\gg_b b^Z$, and the least admissible band satisfies
\begin{equation}\label{eq:zmin}
 Z_{\min}(\delta_0)\gg_{b,d}\frac{C_0}{\sqrt{\delta_0}}.
\end{equation}
Put $\ell_D=\lceil\log_b(C_b'D)\rceil$, with $C_b'$ sufficiently large. Since
$q\ll W^s$,
\begin{equation}\label{eq:ellbound}
 \ell_D\le sL+O(1).
\end{equation}
Partition the remaining segment greedily into completed blocks whose span
first reaches $3\ell_D$. Thus every completed block has span $h$ satisfying
\begin{equation}\label{eq:blockspan}
 3\ell_D\le h\le4\ell_D.
\end{equation}
Discard the last incomplete block, every block whose endpoint has less than
$(\Lambda_{\rm cont}+s+2)L$ of interior continuation, and every block with
$r>4h/Z$ gaps. The incomplete block has span below $3\ell_D$. The blocks
failing the continuation condition lie in a terminal suffix of span at most
$(\Lambda_{\rm cont}+s+2)L+4\ell_D$. By~\eqref{eq:ellbound}, these two losses
together are $O_{b,d}(L)$ and hence at most $G/4$ after $C_0$ is enlarged.
For every block discarded by the last condition, $h<(Z/4)r$. The blocks are
disjoint, so their total span is below
\[
 \frac Z4\sum r\le\frac Z4n\le\frac G4,
\]
where the last inequality uses $Z\le z=G/n$. Thus the retained blocks have
total span at least $G/2$. Since $G\ge V_k/8$, this gives
\begin{equation}\label{eq:cover}
 V_k\ll_{b,d}\sum_{B\ {\rm retained\ in}\ k}\Span(B).
\end{equation}

\begin{lemma}[Interior block words]\label{lem:blocks}
Let $\mathcal N(D,Z)$ be the number of retained block words in the fixed
bands $(D,Z)$.
For every sufficiently large fixed $Z_0$,
\begin{equation}\label{eq:blockentropy}
 \mathcal N(D,Z)\ell_D\ll D^{\tau(Z_0)}
 \qquad(Z\ge Z_0),
\end{equation}
where $\tau(Z_0)\to0$ as $Z_0\to\infty$. A retained block word and its
denominator band determine the exact highest-coefficient state at both block
endpoints.
\end{lemma}

\begin{proof}
By construction, $3\ell_D\le h\le4\ell_D$ and
$1\le r\le16\ell_D/Z$. For fixed $h$ and $r$, there are
$\binom{h-1}{r-1}$ positive compositions. Summing first over $r$ and then over
$h$ gives, for sufficiently large $Z_0$,
\[
 \mathcal N(D,Z)
 \le (4\ell_D+1)^2
 \exp\!\left(C\ell_D\frac{\log(2Z)}Z\right),
\]
uniformly for $D\gg_b b^Z$. Here the square factor accounts for the choices of
$h$ and $r$. We have $\ell_D\asymp_b\log D$ and $\ell_D\gg_b Z$ in these
bands. Enlarging the constant to absorb the polynomial factor, the right-hand
side is at most $D^{\tau(Z_0)}/\ell_D$ for $Z\ge Z_0$, where one may take
\[
 \tau(Z_0)=C_b\sup_{u\ge Z_0}\frac{\log(2u)}u.
\]
This proves~\eqref{eq:blockentropy}, and $\tau(Z_0)\to0$.
For a word of span $h$, iteration of~\eqref{eq:topmap} gives
\[
 \mu_{\rm end}=b^h\mu_{\rm start}-M_B,\qquad M_B\in\Z.
\]
Both endpoints are interior, and the denominator preserved along the block
lies in $[D,2D)$. Thus $\mu_{\rm start}$ lies in an interval of length
$((b-1)b^h)^{-1}$. Since $h\ge3\ell_D$, the choice of $C_b'$ makes this length
smaller than $1/(4D^2)$. Distinct reduced fractions with denominators below
$2D$ are separated by at least $1/(4D^2)$. There is therefore at most one
possible $\mu_{\rm start}$, and the displayed affine relation determines
$\mu_{\rm end}$.
\end{proof}

\begin{lemma}[Sampling and integral fibres]\label{lem:sampling}
Let $J$ be an interval of length $H$.
\begin{enumerate}
\item If a polynomial $f$ of degree at most $d$ satisfies $|f(x)|\le Y$ at
$U\ge2d+1$ distinct integers of $J$, then
\[
 \sup_J|f|\ll_d Y(1+H/U)^d.
\]
\item If $P\in\Q[x]$ has degree $d$ and leading monomial coefficient of
reduced denominator $q_0$, then, with $\sigma=1/s$,
\begin{equation}\label{eq:fibre}
 \#\{x\in J\cap\Z:P(x)\in\Z\}\le d+dHq_0^{-\sigma}.
\end{equation}
\end{enumerate}
\end{lemma}

\begin{proof}
For (1), write the sample points in increasing order and select $d+1$ of them
at ranks nearest to $j(U-1)/d$, $0\le j\le d$. Because distinct integral
sample points are separated by at least their rank difference, the selected
coordinates are mutually separated by $\gg_d U$. Lagrange interpolation on
these $d+1$ nodes gives, for $z\in J$,
\[
 |f(z)|\ll_d Y\left(1+\frac HU\right)^d.
\]

For (2), suppose that the fibre consists of
$x_1<\cdots<x_K$ with $K>d$. The $K-d$ spans
$x_{i+d}-x_i$ have total at most $dH$, since each adjacent gap is counted at
most $d$ times. Hence some $d+1$ consecutive fibre points have total span at
most $dH/(K-d)$. Interpolating their integer values shows that $q_0$ divides
their Vandermonde product. Its absolute value is at most
$(dH/(K-d))^s$, and therefore
\[
 K-d\le dHq_0^{-1/s}.
\]
This is~\eqref{eq:fibre}.
\end{proof}

\begin{lemma}[High-frequency coalescence]\label{lem:coalescence}
Fix a retained block word and dyadic bands $(D,Z)$. Among the exact block-end
carry graphs realized on at least $W^{1/2}$ distinct endpoints, at most one
occurs.
\end{lemma}

\begin{proof}
Let $C,C'$ be two such graphs. By~\eqref{eq:graphden}, there is an integer
$T\ll W^{2s}$ such that $T(C-C')\in\Int_d(\Z)$. Lemma~\ref{lem:blocks}
determines the same exact highest state for $C$ and $C'$. As long as the two
continuations remain interior, each next gap is therefore the unique
interior successor and is common to both graphs. Every such gap is at most
$\ell_D$ by~\eqref{eq:qz} and the definition of $\ell_D$. Take the shortest
common forward word whose span $F$ exceeds $\Lambda_{\rm cont}L$. Then
\[
 \Lambda_{\rm cont}L<F\le\Lambda_{\rm cont}L+\ell_D.
\]
The retention condition gives each endpoint at least
$(\Lambda_{\rm cont}+s+2)L$ of interior continuation, so
\eqref{eq:ellbound} shows that this word is available on both sides.

Let $\widetilde C,\widetilde C'$ be the transformed graphs. Their polynomial
corrections are identical and cancel, whence
\[
 T(\widetilde C-\widetilde C')(x)
 =b^F\bigl[T(C-C')\bigr](x-F).
\]
Translation preserves integer-valued polynomials, and thus
\begin{equation}\label{eq:divisiblegraphs}
 T(\widetilde C-\widetilde C')\in b^F\Int_d(\Z).
\end{equation}

Translate each of the two endpoint sets by $F$, and let $J$ be the convex
hull of their union. Lemma~\ref{lem:localscale} and $F=O(L)$ give
$|J|=O(W)$. Each translated set still has at least $W^{1/2}$ elements, and on
it the corresponding transformed graph takes actual carry values of size
$O(N^d)$. Applying Lemma~\ref{lem:sampling} separately and then using
$T\ll W^{2s}$ yields
\[
 \sup_J|T(\widetilde C-\widetilde C')|
 \ll W^{2s}N^d(1+W^{1/2})^d
 \ll N^{2s+3d/2}.
\]
Since $\Lambda_{\rm cont}=2s+3d/2+2$ and
$F>\Lambda_{\rm cont}L$, the last quantity is smaller than $b^F$ for all
sufficiently large $N$. At every integer of $J$,
\eqref{eq:divisiblegraphs} makes the value a multiple of $b^F$; it must
therefore vanish. The interval contains at least $d+1$ integers for large
$N$, so the polynomial vanishes identically. Finally, the graph transform is
invertible, and hence $C=C'$.
\end{proof}

\begin{proposition}[Interior contribution]\label{prop:interior}
There is a function $\varepsilon_{\rm int}(t)\to0$ as $t\downarrow0$ such
that, uniformly for $N^\theta\le W\le N$, the total span of all interior
windows is at most
$\bigl(\varepsilon_{\rm int}(\delta_0)+o(1)\bigr)mW$.
\end{proposition}

\begin{proof}
For a fixed nonrare prefix, all occurrences remaining interior follow one
gap sequence, since they begin on the same locked graph and every interior
state has at most one interior successor. More precisely, index the gaps in
this common interior continuation by their position after the prefix. A
\emph{cell} consists of the prefix, one of its at most $m$ completed-block
positions, and the dyadic $Z$-band of the ambient window. The prefix and
position determine the exact highest state, hence its reduced denominator,
its band $D$, and the next greedy block word. Transforming the locked prefix
graph along the common gaps also determines the full block-end graph within
the cell. Finally,~\eqref{eq:qz} and $q\ll W^s$ give $Z=O(L)$, so there are
only $O(\log L)$ dyadic $Z$-bands. Consequently,
\begin{equation}\label{eq:cellcount}
 \#\{\text{cells}\}\le N^{\varepsilon_p+o(1)}m.
\end{equation}

For a retained window--block pair $(k,B)$, let $a_i$ be its block endpoint
and put $j=i-k$. The map
\begin{equation}\label{eq:sourcemap}
 (k,B)\longmapsto(a_i,j),\qquad 0\le j\le m,
\end{equation}
is injective: $a_i$ determines $i$, then $j$ determines $k$, and the
canonical partition of that window recovers $B$. Here $i$ is the index in the
fixed global increasing enumeration of $S$, so the first implication uses its
injectivity. Call a cell low frequency
if its endpoint set has fewer than $W^{1/2}$ elements. By
~\eqref{eq:sourcemap}, it contains at most $(m+1)W^{1/2}$ pairs. Since every
retained block has span $O(L)$,~\eqref{eq:cellcount} gives
$\sum_{(k,B)\ {\rm low}}\Span(B)
\ll N^{\varepsilon_p+o(1)}m^2LW^{1/2}=o(mW)$, provided
$\varepsilon_p<\theta/2$; after division by $mW$, the right-hand side is at
most $N^{\varepsilon_p-\theta/2+o(1)}L^2$.

For a high-frequency cell, the highest Newton coefficient at a block endpoint
is $Ar/q$, where $q\asymp D$ and $(r,q)=1$. Its leading monomial coefficient
is $Ar/(qd!)$. If $q_0$ is the reduced denominator of this coefficient, then
\[
 q_0=\frac{qd!}{\gcd(Ar,qd!)}
 \ge\frac q{\gcd(A,q)}
 \ge\frac q{|A|}
 \gg_A D.
\]
Lemma~\ref{lem:sampling} gives $K\ll_d1+WD^{-\sigma}$ for the number $K$ of
endpoints on this graph. Since $K\ge W^{1/2}>2d$ for large $N$, the constant
term can be absorbed and $K\ll_dWD^{-\sigma}$. For fixed bands and a block
word, Lemma~\ref{lem:coalescence} leaves at most one high-frequency graph. By
\eqref{eq:sourcemap}, each endpoint has at most $m+1$ preimages, while each
block has span at most $4\ell_D$. Thus the word contributes
$O(mWD^{-\sigma}\ell_D)$. Summing over the $\mathcal N(D,Z)$ words and using
~\eqref{eq:blockentropy} gives
$\sum_{(k,B)\text{ in }(D,Z)}\Span(B)
\ll mW D^{\tau(Z_0)-\sigma}$.
Choose $Z_0$ so that $\tau(Z_0)<\sigma/2$. Since $D\gg_b b^Z$, the sum over
dyadic $D$ is geometric; extending the finite range of realized bands to
infinity only enlarges it. A second geometric sum gives
\[
 \sum_{\substack{Z\ge Z_{\min}(\delta_0)\\ Z\ {\rm dyadic}}}
 \ \sum_{\substack{D\gg_b b^Z\\ D\ {\rm dyadic}}}
 D^{\tau(Z_0)-\sigma}
 \ll_{b,d}
 \sum_{\substack{Z\ge Z_{\min}(\delta_0)\\ Z\ {\rm dyadic}}}
 b^{-\sigma Z/2}=o_{\delta_0\downarrow0}(1).
\]
Together with the low-frequency estimate and~\eqref{eq:cover}, this proves the
proposition. The last display, multiplied by a fixed constant, may be taken
as $\varepsilon_{\rm int}(\delta_0)$; the low-frequency term is $o(1)$.
\end{proof}

\begin{proof}[Proof of Theorem~\ref{thm:main}]
Retain the preceding choices of $C_0$ and $Z_0$.
Choose $\delta_0>0$ so small that
\[
 \varepsilon_p<\frac\theta2,\qquad
 \varepsilon_p+\varepsilon_e<\theta,\qquad
 2C_0\sqrt{\delta_0}+\varepsilon_{\rm int}(\delta_0)<\frac12,
\]
and $Z_{\min}(\delta_0)\ge Z_0$ in~\eqref{eq:zmin}.

Lemma~\ref{lem:windows} gives total window mass at least $(1-o(1))mW$.
By~\eqref{eq:short},~\eqref{eq:rare}, and Propositions~\ref{prop:exterior}
and~\ref{prop:interior}, the short, rare-prefix/exterior, and interior parts
contribute at most $2C_0\sqrt{\delta_0}\,mW$, $o(mW)$, and
$(\varepsilon_{\rm int}(\delta_0)+o(1))mW$, respectively. The chosen constants
make this impossible for large $N$, so $\delta>\delta_0$.

The bounds are uniform in $W$. For $\varepsilon<\theta$, the endpoint,
rare-prefix/exterior, and low-frequency error ratios satisfy
\[
\begin{aligned}
 \frac{mL}{W}&\le\frac{L^2}{N^\theta}\to0,&
 \frac{N^{\varepsilon+o(1)}L}{W}
   &\le N^{\varepsilon-\theta+o(1)}L\to0,\\
 \frac{N^{\varepsilon_p+o(1)}mL}{W^{1/2}}
   &\le N^{\varepsilon_p-\theta/2+o(1)}L^2\to0.
\end{aligned}
\]
The remaining bounds use only $W\le N$ or $W\ge N^\theta$. Taking the maximum
of their $W$-independent thresholds gives one $N_0$ for the full range, and
$c=\delta_0$ proves the theorem. The choices of $C_0,Z_0$, and $\delta_0$
depend only on $b,w,Q$, and $\theta$; the numerator, the support, and the
$b$-primary part of the denominator enter only through the eventual
thresholds. This also proves the asserted dependence of $c$.
\end{proof}

\section*{Acknowledgments}

The author thanks Stijn Cambie for his help with and contributions to the
restructuring of the manuscript.

\section*{Disclosure of AI assistance and formal verification}

AI tools were used to assist with language editing, proof checking, and the
implementation of the Lean~4 formal verification. The author reviewed the
resulting text and proofs and takes responsibility for all mathematical
claims. The Lean~4 formalization and trust audit of the main density
conclusion, Corollary~\ref{cor:erdos}, and finite counterparts of all numbered
estimates are available at
\href{https://github.com/Hanziwww/erdos260/tree/e708b584c9a3b54857f050d6b7efbecb8a5ea27a}{commit \texttt{e708b58}}.

\end{document}